\documentclass[oneside, reqno]{amsart}

\usepackage{amssymb}

\usepackage{graphicx}

\usepackage{tikz}

\usepackage{cleveref}

\newtheorem{theorem}{Theorem}[section]
\newtheorem{lemma}[theorem]{Lemma}

\newtheorem*{thmarbitrary_values}{Theorem~\ref{thm:arbitrary_values}} 
\newtheorem*{thmalmost_all}{Theorem~\ref{thm:almost_all}} 

\theoremstyle{definition}
\newtheorem{definition}[theorem]{Definition}
\newtheorem{example}[theorem]{Example}

\theoremstyle{remark}
\newtheorem{remark}[theorem]{Remark}

\numberwithin{equation}{section}

\newcommand{\demph}[1]{\emph{#1}}
\newcommand{\Mag}{\mathrm{Mag}}
\newcommand{\Met}{\operatorname{Met}}
\newcommand{\sumofentry}{\operatorname{\mathbf{sum}}}

\renewcommand{\d}{\mathrm{d}}
\renewcommand{\vec}{\mathbf}
\renewcommand{\th}{\text{th}}

\begin{document}

\title{The small-scale limit of magnitude and the one-point property}

\author{Emily Roff}
\address{School of Mathematics, University of Edinburgh, Edinburgh EH9 3FD, Scotland}
\email{emily.roff@ed.ac.uk}

\author{Masahiko Yoshinaga}
\address{Department of Mathematics, Osaka University, Toyonaka, Osaka 560-0043, Japan}
\email{yoshinaga@math.sci.osaka-u.ac.jp}


\date{\today}

\begin{abstract}
The magnitude of a metric space is a real-valued function whose parameter controls the scale of the metric. A metric space is said to have the \emph{one-point property} if its magnitude converges to 1 as the space is scaled down to a point. Not every finite metric space has the one-point property: to date, exactly one example has been found of a finite space for which the property fails. Understanding the failure of the one-point property is of interest in clarifying the interpretation of magnitude and its stability with respect to the Gromov--Hausdorff topology. We prove that the one-point property holds generically for finite metric spaces, but that when it fails, the failure can be arbitrarily bad: the small-scale limit of magnitude can take arbitrary real values greater than 1.
\end{abstract}

\maketitle

\section{Introduction}

Magnitude is an invariant of enriched categories, analogous in a precise sense to Euler characteristic \cite{LeinsterMagnitude2008, LeinsterMagnitude2013}. Every metric space can be regarded as a category enriched in the poset of nonnegative real numbers \cite{LawvereMetric1974}, so magnitude can be interpreted for metric spaces, and in this setting a rich theory has been developed. The magnitude of a metric space defines a partial function \([0, \infty) \to \mathbb{R}\) whose parameter controls the scale of the metric---or, if you prefer, the viewpoint of an observer. The large-scale asymptotics of the magnitude function are well studied, but its behaviour at small scales---what it sees from far away---remains mysterious. This paper investigates that mystery.

Concretely, given a finite metric space \(X\), denote by \(Z_X\) the \(X \times X\) matrix with entries $Z_X(x,y)=e^{-d(x, y)}$. If \(Z_X\) is invertible, the \demph{magnitude} \(|X|\) is defined to be the sum of the entries in the matrix \(Z_X^{-1}\). Now, for each \(t \in [0, \infty)\), let $tX$ be the metric space with underlying set \(X\), in which the distance from \(x\) to \(y\) is $t \cdot d(x, y)$. The \demph{magnitude function} of \((X,d)\) maps \(t\) to \(|tX|\) whenever \(Z_{tX}\) is invertible.

In general, \(|tX|\) is defined for all but finitely many values of \(t\) (see \Cref{rmk:mag_formal_mag}). If the space \(X\) happens to be of \emph{negative type}---equivalently, if the matrix \(Z_{tX}\) is positive definite for every \(t > 0\) \cite[Theorem 3.3]{MeckesPositive2013}---then \(|tX|\) is defined for every \(t\) and satisfies \(|tY| \leq |tX|\) for every \(Y \subseteq X\). Magnitude can thus be extended from finite to compact metric spaces of negative type by defining
\[|tX| = \sup\{|tY| \mid Y \subseteq X \text{ is finite}\}.\]
We will mainly be concerned with the magnitude of finite spaces, not necessarily of negative type.

Magnitude is so-named for a striking series of connections to notions of size and dimension. For any fixed choice of the parameter, magnitude behaves formally like the cardinality of sets: it is multiplicative with respect to \(\ell_1\)-products and satisfies an inclusion-exclusion formula \cite[\S 2]{LeinsterMagnitude2013}. Indeed, as \(t \to \infty\) the magnitude function of a finite space converges to the cardinality of the underlying set \cite[Prop.~2.2.6 (v)]{LeinsterMagnitude2013}. The function as a whole, however, is sensitive not only to the number of points in a space but to the distances between them.

\begin{figure}
\centering
\includegraphics[width=0.8\textwidth]{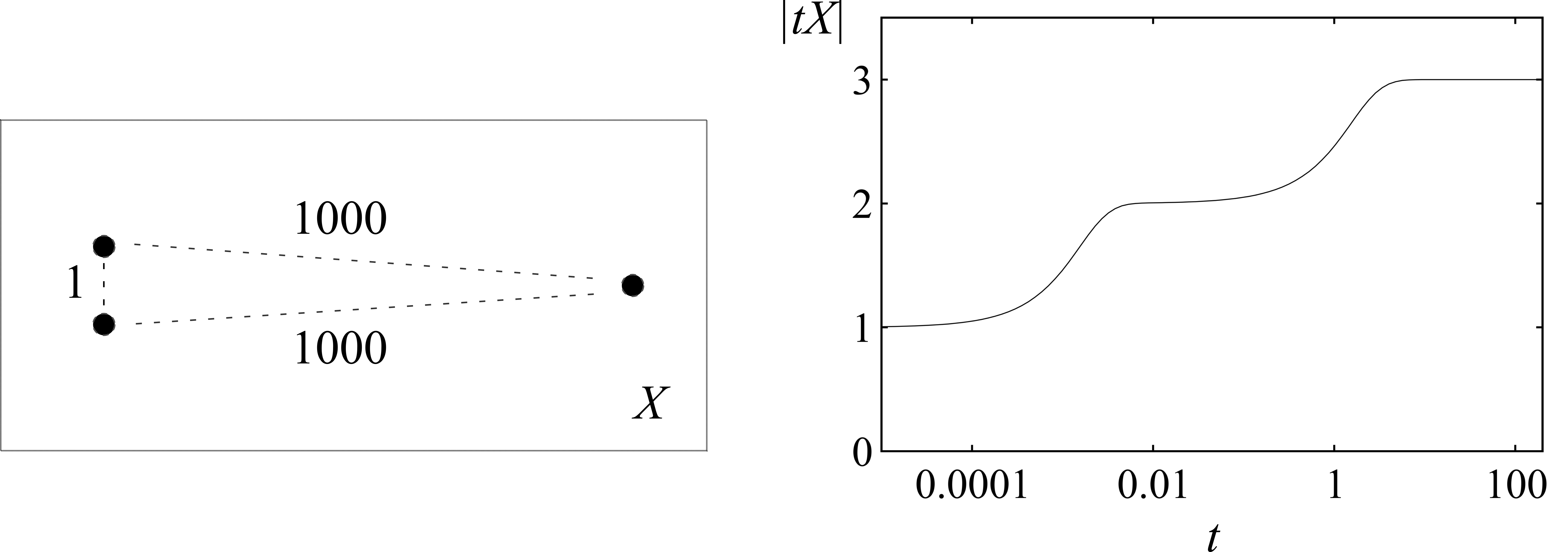}
\caption{The magnitude function of a three-point space (note the logarithmic scale). When \(t\) is very small, \(|tX|\) is close to 1, as though recognizing that `from far away' \(X\) looks like a one-point space. As \(t\) increases, the cluster of two points on the left in \(X\) becomes distinguishable from the point on the right and for a while \(|tX|\) lingers close to 2. As \(t \to \infty\), the magnitude function converges to the cardinality of \(X\): `from close up' it is clear that \(X\) is a three-point space. This is Example 6.4.6 in \cite{LeinsterEntropy2021}, originally due to Willerton.}
\label{fig:3pt}
\end{figure}

The example in \Cref{fig:3pt} illustrates typical behaviour of the magnitude function in a finite setting. Thanks to examples such as this one, the magnitude of a finite metric space is interpreted as measuring the `effective number of points' in the space as the scale (or the viewpoint) varies. Closely related to the magnitude function is the \emph{spread} of a metric space, whose instantaneous rate of growth can be interpreted, similarly, as measuring `effective dimension' \cite[\S 4]{WillertonSpread2015}.

For compact metric spaces of negative type, magnitude turns out to capture information about many more classical size-related features. In recent years, the large-scale asymptotics of the magnitude function have been the primary focus of attention. For example, for a compact subset \(X\) of Euclidean space the large-\(t\) asymptotics of \(|tX|\) have been shown to determine the volume of \(X\) and its Minkowski dimension (\cite[Theorem 1]{BarceloCarbery2018} and \cite[Cor.~7.4]{MeckesMagnitude2015}). Under additional conditions they also record the surface area, total mean curvature, and the Willmore energy (\cite[Theorem 2(d)]{GimperleinGoffeng} and \cite[Theorem 2]{GimperleinGoffengWillmore}).

By contrast, little is known---even in the finite setting---about the behaviour of the magnitude function at small scales. At the heart of the mystery is the so-called `one-point property'.

\subsection*{The one-point property} As one would hope, the magnitude of a one-point space is 1. Let \(X\) be a metric space which is compact and negative type, or finite. We say that \(X\) has the \demph{one-point property} if
\[\lim_{t \to 0} |tX| = 1.\]
Given the interpretation of magnitude as a scale-dependent measure of the effective number of points in a space, one might expect that \emph{every} compact metric space has the one-point property: that viewed from very far away, every such space is `effectively' a single point. However, this is not quite the case.

Indeed, Leinster and Meckes have recently exhibited a class of compact spaces \(X\) such that \(|tX| = \infty\) for every \(t > 0\) \cite[Theorem 2.1]{LeinsterMeckes2021}. And even for a finite space---which necessarily satisfies \(|tX| < \infty\) for all sufficiently small \(t > 0\)---the one-point property can fail. To date, essentially just one example has been found of a finite metric space without the one-point property (Example 2.2.8 in \cite{LeinsterMagnitude2013}, due to Willerton; see \Cref{eg:willerton}). It is a six-point space of negative type, satisfying
\[\lim_{t \to 0} |tX| = 6/5.\]
Other examples can be constructed from this one by taking \(\ell_1\)-powers, which yields, for each \(n \in \mathbb{N}\), a space \(X^n\) such that \(\lim_{t \to 0} |tX^n| = (6/5)^n\).

This failure of the one-point property troubles the interpretation of magnitude as `effective number of points'. It is also one of just a handful of known discontinuities of magnitude, regarded as a function on the Gromov--Hausdorff space of compact metric spaces of negative type (see \cite[Prop.~3.1]{LeinsterMagnitude2017b} and \cite[Example 2.3]{GimperleinGoffengLouca2022}). Thus, it is of both conceptual and practical importance to understand this apparently pathological behaviour.

Three natural questions present themselves:
\begin{enumerate}
    \item How \emph{commonly} does the one-point property fail?
    \item How \emph{soon} does the one-point property fail? In other words, what is the smallest set that carries a metric for which the property fails?
    \item How \emph{badly} can the one-point property fail? What values can the small-scale limit of the magnitude function take?
\end{enumerate}
This paper gives answers to all three questions in the setting of finite metric spaces.

\subsection*{Summary of results}

The strongest prior result concerning the one-point property is due to Leinster and Meckes \cite[Theorem 3.1]{LeinsterMeckes2021}. Let \(L_1[0,1]\) denote the Banach space of measurable functions \(f : [0,1] \to \mathbb{R}\) whose integral 1-norm  \(\int |f|\) is finite, with the metric induced by the 1-norm. Leinster and Meckes prove that every nonempty compact subset of a finite-dimensional subspace of \(L_1[0,1]\) has the one-point property. This implies in particular that the one-point property holds for every compact subset of \(\ell_1^n\) or \(\mathbb{R}^n\) with the Euclidean metric.

For various other classes of finite spaces the one-point property is guaranteed to hold by virtue of known formulae for magnitude: these include all finite complete graphs, complete bipartite graphs, cycles and Cayley graphs. In \Cref{sec:almost_all} we record these examples, before proving---in answer to question (1)---that in fact a \emph{generic} finite metric space has the one-point property:

\begin{thmalmost_all}
The space of all \(n\)-point metric spaces contains a dense open subset on which the one-point property holds.
\end{thmalmost_all}

In \Cref{sec:any_value} we turn to questions (2) and (3). Since every metric space with at most four points embeds isometrically into \(\ell_1^2\) \cite[Theorem 3.6 (4)]{MeckesPositive2013}, Leinster and Meckes's result implies that every space with at most four points has the one-point property. Thus, question (2) comes down to asking whether there exists a five-point space without the one-point property. In \Cref{eg:5_pt} we exhibit one.

Finally, in answer to question (3), we show that, although the one-point property almost never fails, the failure can be arbitrarily bad, in the sense that the small-scale limit of magnitude can take arbitrary real values greater than 1. That is:

\begin{thmarbitrary_values}
    For every real number \(R \geq 1\) there exists a finite metric space \(X\) such that \(\lim_{t \to 0} |tX| = R\).
\end{thmarbitrary_values}

Restricting attention to finite metric spaces allows us to employ elementary methods throughout. At least three fundamental questions are left open:
\begin{enumerate}
\item
For a nonempty metric space \(X\) of negative type, the monotonicity of magnitude with respect to inclusion implies that \(\lim_{t \to 0} |tX| \geq 1\). In general, though, magnitude can take values below 1 \cite[Example 2.2.7]{LeinsterMagnitude2013}. Does there exist a metric space \(X\) such that \(\lim_{t \to 0} |tX| < 1\)?
\item
For a finite metric space \(X\), can \(\lim_{t \to 0} |tX|\) be bounded in terms of the cardinality of \(X\)?
\item
Does there exist a finite metric space \(X\) with \(\lim_{t \to 0} |tX| = \infty\)?
\end{enumerate}

\subsection*{Acknowledgements}

This work was partially supported by JSPS Postdoctoral Fellowships for Research in Japan and JSPS KAKENHI JP22K18668. We are grateful to Jun O'Hara for useful comments, and to the anonymous referee for several very helpful remarks, including the observation that taking \(\ell_1\)-powers of Willerton's example generates a family of spaces whose small-scale magnitude can be arbitrarily large.


\section{A generic finite metric space has the one-point property}
\label{sec:almost_all}

We begin by describing an alternative formulation for magnitude, introduced by Leinster to study the magnitude of graphs \cite[\S 2]{LeinsterMagnitude2019} and later employed in the construction of \demph{magnitude homology}, a homology theory for metric spaces designed to categorify their magnitude \cite[Example 2.5]{LeinsterMagnitude2017v3}. Here, and throughout, when we refer to the magnitude of a graph we mean the magnitude of the set of vertices equipped with the shortest path metric.

Let $\mathbb{Z}[q^{\mathbb{R}_{\geq 0}}]$ denote the ring of \emph{generalized polynomials}: finite sums of the form 
\[
a_1q^{r_1}+a_2q^{r_2}+\cdots+a_nq^{r_n}
\]
for some $a_i \in \mathbb{Z}$ and $r_i\in [0, \infty)$, with multiplication determined by \(q^{r_1} \cdot q^{r_2} = q^{r_1 + r_2}\). This is an integral domain, and we denote its field of fractions---the field of \demph{generalized rational functions}---by $\mathbb{Q}(q^{\mathbb{R}})$.

Given a finite metric space \(X\), let \(Z_X^{\mathrm{form}}\) denote the \(X \times X\) matrix with entries \(Z_X^{\mathrm{form}}(x,y) = q^{d(x,y)}\). The determinant of this matrix is a generalized polynomial with constant term 1, so is invertible in \(\mathbb{Q}(q^\mathbb{R})\). The \emph{formal magnitude} \(\Mag(X)\) is defined to be the sum of entries in the matrix \((Z_X^{\mathrm{form}})^{-1}\).

\begin{remark}\label{rmk:mag_formal_mag}
    The magnitude function of a finite metric space \(X\) can be recovered from the formal magnitude by 
    \[
    |tX|=\Mag(X)(e^{-t})
    \]
    (see \cite[Example 2.5]{LeinsterMagnitude2017v3}), and the space \(X\) has the one-point property if and only if
    \[\lim_{q \to 1} \Mag(X)(q) = 1.\]
    Since the determinant of \(Z_X^{\mathrm{form}}\) is a generalized polynomial with constant term 1, it vanishes for at most finitely many values of \(q\), and certainly not at \(q=0\). It follows that \(\det(Z_{tX}) = \det(Z_X^{\mathrm{form}})|_{q = e^{-t}}\) is nonzero---thus, \(|tX|\) is defined---for all but finitely many \(t\), including for all sufficiently large \(t\) and sufficiently small \(t > 0\).
\end{remark}

\begin{example}\label{eg:one-pt_examples}
For various classes of spaces, the one-point property is guaranteed by known results.
    \begin{enumerate}
        \item Every finite metric space \(X\) that can be embedded isometrically into Euclidean space of some finite dimension has the one-point property by Leinster and Meckes's result \cite[Theorem 3.1]{LeinsterMeckes2021}. Schoenberg's Criterion \cite[Theorem 1]{Schoenberg1935} tells us that such an embedding exists if and only if the matrix
        \(D = \left(d(x,x')^2\right)_{x,x' \in X}\)
        is conditionally negative semidefinite.
        
        \item \label{eg:trees} Trees which are not path graphs cannot be embedded into Euclidean space---but every finite tree can be embedded into \(\ell_1^n\) for some \(n \in \mathbb{N}\). Thus, Leinster and Meckes's result guarantees that every finite tree has the one-point property. This can also be seen directly from Example 4.12 in \cite{LeinsterMagnitude2019}, which says that the formal magnitude of a forest \(F\) is given by the formula
        \[\Mag(F)(q) = \#\{\text{components of \(F\)}\} + \#\{\text{edges in \(F\)}\} \frac{1-q}{1+q}.\]
        (Here, magnitude has been extended from classical metric spaces to those with distances valued in \([0,\infty]\), using the convention that \(q^{\infty} = 0\).) It follows, via \Cref{rmk:mag_formal_mag}, that
        \[\lim_{t \to 0} |tF| = \lim_{q \to 1} \Mag(F)(q) = \#\{\text{components of \(F\)}\}\]
        and in particular that if $F$ is a tree, then it has the one-point property. 
        
        \item \label{eg:bipartite} Example 3.4 in \cite{LeinsterMagnitude2019} says that the formal magnitude of the complete bipartite graph \(K_{n,m}\) is given by the formula
        \[\Mag(K_{n,m})(q) = \frac{(m+n) - (2mn - m - n)q}{(1+q)(1-(m-1)(n-1)q^2)}.\]
        It follows that every complete bipartite graph has the one-point property.
        
        \item \label{eg:homogeneous} A metric space is called \demph{homogeneous} if its group of isometries acts transitively. Given a finite homogeneous metric space $X$ and a fixed point \(x \in X\), define the generalized polynomial
        \[N_X(q) = \frac{1}{\# X} \sum_{x' \in X} q^{d(x,x')}.\]
        The homogeneity of \(X\) implies that \(N_X(q)\) does not depend on the choice of \(x\). Speyer's formula for the magnitude of a finite homogeneous space \cite[Prop.~2.1.5]{LeinsterMagnitude2013} tells us that
        \[\Mag(X)(q)=\frac{1}{N_X(q)}.\]
        This formula implies the one-point property for all such spaces, including all vertex-transitive graphs. It follows that the property holds for every finite complete graph, every finite cycle, and every finite Cayley graph.
    \end{enumerate}
\end{example}

In fact, \textit{almost every} finite metric space has the one-point property. To make this statement precise, we first consider \(n\)-point metric spaces equipped with an ordering on their points, and call two such spaces isomorphic if there exists an order-preserving isometry between them. Each ordered metric space $X=(\{x_1, \dots, x_n\}, d)$ determines a \emph{distance matrix} $\mathbf{d} = (d(x_i, x_j))_{i, j=1, \dots, n}$; conversely, the distance matrix determines the isomorphism class of $X$. Thus the isomorphism classes of ordered \(n\)-point metric spaces are parametrized by the set 
\[
\label{eq:matrix}
\Met_n=
\left\{
(d_{ij})_{i, j=1, \dots, n}
\left|\ 
\begin{aligned}
    &d_{ii}=0\ (1\leq i\leq n)\\
&d_{ij}=d_{ji}>0\ (1\leq i<j\leq n)\\
&d_{ij}+d_{jk}\geq d_{ik}\ (1\leq i, j, k\leq n)
\end{aligned}
\right.
\right\}.
\]
This is a $\binom{n}{2}$-dimensional convex cone in Euclidean space (neither open nor closed).

Two spaces $X=\{x_1, \dots, x_n\}$ and $Y=\{y_1, \dots, y_n\}$ are isometric if and only if there exists a permutation $\sigma\in\mathfrak{S}_n$ such that $d_X(x_{\sigma(i)}, x_{\sigma(j)})=d_Y(y_i, y_j)$. Thus the set of isometry classes of (unordered) $n$-point metric spaces can be identified with the quotient set $\Met_n/\mathfrak{S}_n$. We equip this set with the quotient topology, which is equivalent to the Gromov--Hausdorff topology \cite[Definition 5.33]{BridsonHaefliger}, and call this \demph{the space of \(n\)-point metric spaces}.

\begin{theorem}\label{thm:almost_all}
    The space of \(n\)-point metric spaces contains a dense open subset on which the one-point property holds.
\end{theorem}

\begin{proof}
We will show that there exists a non-zero polynomial $F_n(d_{ij})$ on the cone $\Met_n$ such that if $F_n(d(x_i, x_j))\neq 0$, then $X= (\{x_1, \dots, x_n\},d)$ has the one-point property. It follows that the set of ordered \(n\)-point spaces with the one-point property contains a dense open subset of $\Met_n$, and since the quotient map $\Met_n\to\Met_n/\mathfrak{S}_n$ is surjective and an open map, it descends to $\Met_n/\mathfrak{S}_n$.

Let \(X\) be an \(n\)-point metric space. By taking \(t\) to be sufficiently small, we may assume that $\det (Z_{tX})\neq 0$ and \(|tX|\) is defined (\Cref{rmk:mag_formal_mag}). By Cramer's rule,
\begin{equation}
\label{eq:cofactor}
|tX| = \frac{\sumofentry (\mathrm{adj}(Z_{tX}))}{\det (Z_{tX})}
\end{equation}
where $\mathrm{adj}(Z_{tX})$ is the adjugate matrix of $Z_{tX}$ and \(\sumofentry(M)\) denotes the sum of all entries in a matrix \(M\).

We first consider the denominator of (\ref{eq:cofactor}). Let $\mathbf{d}_i=(d_{1i},d_{2i}, \dots, d_{ni})^{\top}$ be the \(i^\th\) column of the distance matrix $\mathbf{d}=(d_{ij})=(d_X(x_i, x_j))$, and for $k\geq 0$ let $\mathbf{d}_i^{\odot k}=(d_{1i}^k,d_{2i}^k, \dots, d_{ni}^k)^\top$ be the \(k^\th\) Hadamard power (component-wise power).
Note that 
$\mathbf{d}_i^{\odot 0}=(1, 1, \dots, 1)^\top$. 
Then the \(i^\th\) column of the matrix $Z_{tX}$ is 
$
\sum_{k=0}^\infty \frac{(-t)^k}{k!}\mathbf{d}_i^{\odot k}. 
$
Hence, we have 
\[
\label{eq:expand}
\begin{split}
\det (Z_{tX})&=
\sum_{k_1, \dots, k_n=0}^\infty \frac{(-t)^{k_1+\dots +k_n}}{k_1!\cdots k_n!}
\det
\begin{pmatrix}
\mathbf{d}_1^{\odot k_1}, 
\mathbf{d}_2^{\odot k_2}, 
\cdots, 
\mathbf{d}_n^{\odot k_n}
\end{pmatrix}
\\
&=
\sum_{p=0}^\infty 
\left(
\sum_{k_1+ \dots +k_n=p} \frac{1}{k_1!\cdots k_n!}
\det
\begin{pmatrix}
\mathbf{d}_1^{\odot k_1}, 
\mathbf{d}_2^{\odot k_2}, 
\cdots, 
\mathbf{d}_n^{\odot k_n}
\end{pmatrix}
\right)
(-t)^{p}.
\end{split}
\]
If $p=k_1+\dots+k_n\leq n-2$, at least two of $k_1, \dots, k_n$ are zero. 
In that case, the determinant is zero because there are 
at least two identical columns $(1, 1, \dots, 1)^\top$. 
Thus, the first potentially non-vanishing 
term is $(-t)^{n-1}$, with the coefficient 
\begin{equation}
\label{eq:initialcoeff}
\sum_{j=1}^n
\det
\begin{pmatrix}
\mathbf{d}_1^{\odot 1}, 
\cdots, 
\mathbf{d}_j^{\odot 0}, 
\cdots, 
\mathbf{d}_n^{\odot 1}
\end{pmatrix}
.
\end{equation}
Here, the \(j^\th\) matrix whose determinant appears in the sum is obtained by replacing the \(j^\th\) column in the distance matrix by the vector \((1, \ldots, 1)^{\top}\).
So \eqref{eq:initialcoeff} is a polynomial in the entries of the matrix \(\mathbf{d}\), which we denote by $F_n(d_{ij})$. 

Next we consider the numerator of (\ref{eq:cofactor}). In general, 
the sum of entries of the adjugate matrix of a square matrix $A=(a_{ij})_{i,j}$ is 
given by 
\[
\sum_{i=1}^n
\det
\begin{pmatrix}
a_{11}& \cdots& 1& \cdots & a_{1n}\\
\vdots&\ddots &\vdots& \ddots &\vdots\\
a_{n1}& \cdots &1& \cdots& a_{nn}
\end{pmatrix},
\]
where the \(i^\th\) matrix whose determinant appears in the formula is obtained by replacing the \(i^\th\) column in \(\mathrm{adj}(A)\) by the vector \((1, \ldots, 1)^{\top}\).
Hence,
$\sumofentry (\mathrm{adj}(Z_{tX}))$ is equal to 
\[
\sum_{i=1}^n
\sum_{k_1, \dots, k_{i-1}, k_{i+1}, \dots, k_n=0}^\infty \frac{(-t)^{k_1+\dots +k_n}}{k_1!\cdots k_n!}
\det
\begin{pmatrix}
\mathbf{d}_1^{\odot k_1}, 
\cdots, 
\mathbf{d}_i^{\odot 0}, 
\cdots, 
\mathbf{d}_n^{\odot k_n}
\end{pmatrix}. 
\]
Again the initial term is $t^{n-1}$ with coefficient equal to 
(\ref{eq:initialcoeff}). Thus we have the expression 
\begin{equation}
\label{expansion}
|tX|=\frac{F_n(d(x_i, x_j))(-t)^{n-1}+C_n (-t)^n+\dots}{F_n(d(x_i, x_j))(-t)^{n-1}+C'_n (-t)^n+\dots},
\end{equation}
from which we can see that if $F_n(d(x_i, x_j))\neq 0$, then $\lim_{t\to 0}|tX|=1$. 

It remains to prove that $F_n$ is not identically zero on the cone $\Met_n$, i.e.~that there exists a metric space $X$ with $F_n(d(x_i, x_j))\neq 0$. Let $X$ be the \(n\)-point metric space with $d(x_i, x_j)=1$ for $x_i\neq x_j$. Consider the expression \eqref{eq:initialcoeff} defining \(F_n\). 
Using 
\[
\det
\begin{pmatrix}
0&1&\dots &1&1\\
1&0&\dots &1&1\\
\vdots&\vdots&\cdots &\vdots&\vdots\\
1&1&\dots &1&1\\
\vdots&\vdots&\cdots &\vdots&\vdots\\
1&1&\dots &0&1\\
1&1&\dots &1&0
\end{pmatrix}^\top
=
\det
\begin{pmatrix}
-1&0&\dots &0&0\\
0&-1&\dots &0&0\\
\vdots&\vdots&\cdots &\vdots&\vdots\\
1&1&\dots &1&1\\
\vdots&\vdots&\cdots &\vdots&\vdots\\
0&0&\dots &-1&0\\
0&0&\dots &0&-1
\end{pmatrix}^\top
=(-1)^{n-1}, 
\]
we see that $F_n(d(x_i, x_j))=(-1)^{n-1}\cdot n\neq 0$. 
\end{proof}

\begin{remark}\label{rmk:Fn}
The proof of \Cref{thm:almost_all} shows that the relation $F_n(d(x_i,x_j))=0$ is necessary for the one-point property to fail. However, this relation does not guarantee failure: for instance, the cycle graph on four vertices satisfies $F_n(d(x_i,x_j))=0$ and has the one-point property. 

On the other hand, formula (\ref{expansion})
shows that when $F_n(d(x_i, x_j))=0$ and \(C_n' \neq 0\), the small-scale limit is
\[\lim_{t\to 0}|tX|=\frac{C_n}{C_n'}.\]
The coefficients $C_n$ and $C_n'$ in (\ref{expansion}) can be computed as
\[
\begin{split}
C_n&=
\frac{1}{2}\cdot
\sum_{\substack{i,j=1\\ i\neq j}}^n
\det
\begin{pmatrix}
\mathbf{d}_1^{\odot 1}, 
\cdots, 
\mathbf{d}_i^{\odot 0}, 
\cdots, 
\mathbf{d}_j^{\odot 2}, 
\cdots, 
\mathbf{d}_n^{\odot 1}
\end{pmatrix}, 
\\
C_n'&=C_n+\det(d(x_i, x_j)). 
\end{split}
\]
Thus, if $F_n(d(x_i, x_j))= 0$ and we also have $\det(d(x_i, x_j))\neq 0$, then the one-point property fails. In \Cref{sec:any_value} we will investigate this phenomenon more closely for a special class of metric spaces. 
\end{remark}


\section{The small-scale limit can take any real value greater than 1}
\label{sec:any_value}

Our aim in this section is to exhibit an infinite family of finite metric spaces for which the one-point property fails, and for which the value of the small-scale limit can be controlled. These spaces will be constructed by generalizing the only previously known example of a finite space without the one-point property. That example, due to Willerton, is the following.

\begin{example}
\label{eg:willerton}
(\cite[Example 2.2.8]{LeinsterMagnitude2013}) 
Let $X$ be the graph in \Cref{fig:Willerton}. Then \(\Mag_X(q) = \frac{6}{4q + 1}\) and so
\[\lim_{t\to 0}|tX| = \lim_{q \to 1} \Mag_X(q) = \frac{6}{5}.\]
\end{example}

\begin{figure}[htbp]
\centering
\begin{tikzpicture}

\coordinate (A1) at (0,2); 
\coordinate (A2) at (1.5,2); 
\coordinate (A3) at (3,2); 
\coordinate (B1) at (0,0); 
\coordinate (B2) at (1.5,-1); 
\coordinate (B3) at (3,0); 

\draw [very thick] (A1)--(B1)--(A2)--(B2)--(A1)--(B3)--(A2);
\draw [very thick] (B1)--(A3)--(B2)--(B1)--(B3)--(B2)--(B3)--(A3);

\filldraw[draw=black, fill=black] (A1) circle [radius=0.1]; 
\filldraw[draw=black, fill=black] (A2) circle [radius=0.1]; 
\filldraw[draw=black, fill=black] (A3) circle [radius=0.1]; 
\filldraw[draw=black, fill=black] (B1) circle [radius=0.1]; 
\filldraw[draw=black, fill=black] (B2) circle [radius=0.1]; 
\filldraw[draw=black, fill=black] (B3) circle [radius=0.1]; 

\end{tikzpicture}
\caption{A six-point space without the one-point property. Solid black lines represent distances equal to 1. The distance between each non-adjacent pair of vertices is $2$.}
\label{fig:Willerton}
\end{figure}
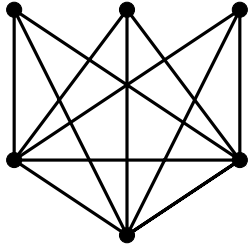

Recall that the \demph{join} of graphs \(G\) and \(H\) is the graph \(G \ast H\) obtained by first taking the disjoint union of \(G\) and \(H\), then adding an edge between each vertex of \(G\) and every vertex of \(H\). Willerton's example is the graph \(D_3 \ast K_3\), where \(D_3\) is the discrete graph on three vertices and \(K_3\) is the complete graph on three vertices. To construct our family of examples we will consider more general joins: not just of graphs, but of metric spaces.

The join of metric spaces extends the construction for graphs---see, for example, \cite[Section 8]{BeardonRV} for the general definition and discussion. For present purposes it suffices to consider joins of spaces of diameter at most 2, defined as follows.

\begin{definition}
    Let \((X,d_X)\) and \((Y,d_Y)\) be metric spaces of diameter at most 2. The \demph{join} of of \(X\) and \(Y\) is the metric space \((X \ast Y, d)\) with underlying set \(X \sqcup Y\) and distance function
    \[d(a,b) = \begin{cases} 
    d_X(a,b) & \text{if } a,b \in X \\
    d_Y(a,b) & \text{if } a,b \in Y \\
    1 & \text{otherwise}.
    \end{cases}\]
\end{definition}

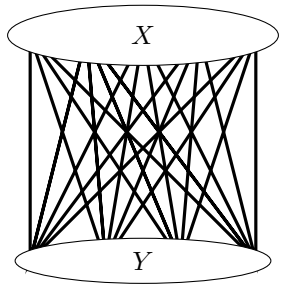
\begin{figure}[htbp]
\centering
\begin{tikzpicture}

\coordinate (A1) at (0,3); 
\coordinate (A2) at (0.75,3); 
\coordinate (A3) at (1.5,3); 
\coordinate (A4) at (2.25,3); 
\coordinate (A5) at (3,3); 
\coordinate (B1) at (0,0); 
\coordinate (B2) at (1,0); 
\coordinate (B3) at (2,0); 
\coordinate (B4) at (3,0); 

\draw [very thick] (B1)--(A1)--(B2)--(A2)--(B1)--(A3)--(B2)--(A2)--(B1)--(A4)--(B2)--(A5)--(B1);
\draw [very thick] (B3)--(A1)--(B4)--(A2)--(B3)--(A3)--(B4)--(A2)--(B3)--(A4)--(B4)--(A5)--(B3);

\filldraw[draw=black, fill=white] (1.5,0) node {$Y$} circle [x radius=1.7,y radius=0.3];
\filldraw[draw=black, fill=white] (1.5,3) node {$X$} circle [x radius=1.8,y radius=0.4];

\end{tikzpicture}
\caption{The join $X*Y$. Solid black lines represent distances equal to 1. In principle one can consider joining distances other than 1, but for us this would not offer extra generality, since we always consider \(X \ast Y\) along with its rescalings \(t(X \ast Y)\) for \(t \in [0, \infty)\).}
\label{fig:join}
\end{figure}

For example, let \(X_2^m\) denote the \(m\)-point space with all nonzero distances equal to 2; then \(X_2^m\) is homogeneous, and for each \(n,m \in \mathbb{N}\) the join \(X_2^m \ast X_2^n\) is isometric to the complete bipartite graph \(K_{m,n}\). In \Cref{eg:one-pt_examples} \eqref{eg:bipartite} we noted that there is a general formula for the magnitude of a complete bipartite graph, due to Leinster. The following theorem generalizes that formula to describe the magnitude of a join of two finite homogeneous spaces---it is essentially the same as \cite[Prop.~2.3.13]{LeinsterMagnitude2013}, but for completeness we include a proof in this paper's notation.

Recall that \(N_X(q) = \frac{1}{\#X} \sum_{x' \in X} q^{d(x,x')}\) for any choice of \(x \in X\).

\begin{theorem}\label{thm:fmag_hom_join}
    Let \(X\) and \(Y\) be finite, nonempty homogeneous metric spaces of diameter at most 2. Then the formal magnitude of \(X \ast Y\) is given by
    \[\Mag(X \ast Y)(q) = \frac{N_X(q) + N_Y(q) - 2q}{N_X(q)N_Y(q) - q^2}.\]
\end{theorem}

\begin{proof}
    Let \(X\) be a finite metric space. A vector \(\vec{w} \in \mathbb{Q}(q^{\mathbb{R}})^X\) is called a \demph{formal weighting} on \(X\) if \((1, \ldots, 1) = Z_X^{\mathrm{form}} \vec{w}\). Since \(Z_X^\mathrm{form}\) is always invertible, every finite space carries a unique formal weighting \(\vec{w}\), and we have \(\Mag(X)(q) = \sum_{x \in X} w_x\).

    Since \(X\) and \(Y\) are homogeneous, to find the formal weighting on the space \(X \ast Y\) it suffices to find a solution \((\alpha, \beta)\) to the system
    \begin{equation}\label{eq:wt_reduced}
    \begin{pmatrix}
    1\\
    1
    \end{pmatrix}
    =
    \begin{pmatrix}
    N_X(q) & q\\
    q & N_Y(q)
    \end{pmatrix}
    \begin{pmatrix}
    \#X \alpha\\
    \#Y \beta
    \end{pmatrix}.
    \end{equation}
    The formal weighting \(\vec{w}\) is then given by
    \[w_v = \begin{cases} \alpha & v \in X \\ \beta & v \in Y,\end{cases}\]
    so \(\Mag(X \ast Y)(q) = \#X\alpha + \#Y\beta\). Solving \eqref{eq:wt_reduced} yields
    \[\#X \alpha = \frac{N_Y(q) - q}{N_X(q)N_Y(q) - q^2} \; \text{ and } \; \#Y \beta = \frac{N_X(q) - q}{N_X(q)N_Y(q) - q^2}\]
    and summing these gives the result.
\end{proof}

As a first application, we exhibit a five-point metric space without the one-point property. Since every four-point space has the one-point property, this is a `smallest' space for which the property fails.

\begin{example}[A five-point space without the one-point property]
\label{eg:5_pt}
Let \(X\) be the space with two points separated by distance \(\frac{4}{3}\), and \(Y\) the space with three points  separated pairwise by distance \(2\). Their join \(X \ast Y\) is shown in \Cref{fig:1st_example}. We have $N_X(q)=\frac{1+q^{4/3}}{2}$ and $N_Y(q)=\frac{1+2q^2}{3}$, so \Cref{thm:fmag_hom_join} gives the formula
\[
\begin{split}
\Mag(X \ast Y)(q)
&=
\frac{5-12q+3q^{4/3}+4q^2}{(1+q^{4/3})(1+2q^2)-6q^2}\\
&=
\frac{(1-q^{1/3})^2(5+10q^{1/3}+15q^{2/3}+8q+4q^{4/3})}{(1-q^{1/3})^2(1+q^{1/3})^2(1+2q^{2/3}+4q^{4/3}+2q^2)}.
\end{split}
\]
It follows that
\[
\lim_{t \to 0} |t(X \ast Y)| = \lim_{q\to 1}\Mag(X \ast Y)(q)=\frac{7}{6}. 
\]
\end{example}

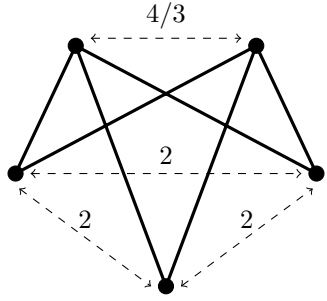
\begin{figure}[htbp]
\centering
\begin{tikzpicture}

\coordinate (A1) at (.3,1.7); 
\coordinate (A2) at (2.7,1.7); 
\coordinate (B1) at (-.5,0); 
\coordinate (B2) at (1.5,-1.5); 
\coordinate (B3) at (3.5,0); 

\draw [very thick] (A1)--(B1)--(A2)--(B2)--(A1)--(B3)--(A2);
\draw[<->, dashed] (.5,1.8) -- node[above] {$4/3$} (2.5,1.8);
\draw[<->, dashed] (-.45,-.2) -- node[above] {$2$} (1.3,-1.5);
\draw[<->, dashed] (1.7,-1.5) -- node[above] {$2$} (3.45,-.2);
\draw[<->, dashed] (-.3,0) -- node[above] {$2$} (3.3,0);
\filldraw[draw=black, fill=black] (A1) circle [radius=0.1]; 
\filldraw[draw=black, fill=black] (A2) circle [radius=0.1]; 
\filldraw[draw=black, fill=black] (B1) circle [radius=0.1]; 
\filldraw[draw=black, fill=black] (B2) circle [radius=0.1]; 
\filldraw[draw=black, fill=black] (B3) circle [radius=0.1]; 

\end{tikzpicture}
\caption{A five-point space without the one-point property. Solid black lines represent distances equal to 1.}
\label{fig:1st_example}
\end{figure}

Our objective is to prove that the small-scale limit of the magnitude function can take any real value greater than 1, and we will achieve this using joins of homogeneous spaces. Many such joins have the one-point property: for instance, the property holds for all complete bipartite graphs. So, in order to find a family for which the small-scale limit can be greater than 1, we impose an additional condition on the spaces involved. That condition---equation \eqref{eq:assumption} below---ensures that the distance matrix of \(X \ast Y\) satisfies the relation $F_n(d_{ij})=0$ which is necessary in order for the one-point property to fail (\Cref{rmk:Fn}).

\begin{theorem}\label{thm:mag1_finite}
Let \(X = \{x_1, \ldots, x_n\}\) and \(Y = \{y_1, \ldots, y_m\}\) be homogeneous metric spaces of diameter at most 2. Suppose that 
\begin{equation}\label{eq:assumption}
    \frac{1}{n} \sum_{i=1}^n d(x_1,x_i) + \frac{1}{m} \sum_{i=1}^m d(y_1,y_i) = 2.
\end{equation}
Then 
\[
\label{eq:explicit_lim}
\lim_{t\to 0} |t (X \ast Y)| = 
\frac
{\frac{1}{n} \sum_{i=1}^n \left( d(x_1,x_i)^2 - d(x_1,x_i) \right) + \frac{1}{m} \sum_{i=1}^m \left( d(y_1,y_i)^2 - d(y_1,y_i) \right)}
{\frac{1}{n^2} \sum_{i < j} (d(x_1,x_i) - d(x_1,x_j))^2 + \frac{1}{m^2} \sum_{i < j} (d(y_1,y_i) - d(y_1,y_j))^2}.
\]
Moreover, we have 
\(1 \leq \lim_{t\to 0} |t (X \ast Y)| < \infty\).
\end{theorem}

Note that the homogeneity of \(X\) and \(Y\) means the sums in the statement of \Cref{thm:mag1_finite} do not depend on the labelling of the points in either space.

In what follows, we denote \(\frac{\d}{\d q} N_X(q)\) by \(N_X'(q)\) and \(\frac{\d^2}{\d q^2} N_X(q)\) by  \(N_X''(q)\). The proof of \Cref{thm:mag1_finite} makes use of two lemmas, the first of which is immediate from the definition of \(N_X\).

\begin{lemma}\label{lem:derivatives}
    For any finite homogeneous space \(X\) we have 
    \begin{align*}
        & N_X(1) = 1, \\
        & N_X'(1) = \frac{1}{\# X} \sum_{x' \in X} d(x,x'), \\
        & N_X''(1)= \frac{1}{\# X} \sum_{x' \in X} (d(x,x')^2 - d(x,x')). \tag*{\qed}
    \end{align*}
\end{lemma}

The second lemma is an application of \demph{Lagrange's identity}, which says that every pair of vectors \(\vec{u}, \vec{v} \in \mathbb{R}^n\) satisfies
\[\left( \sum_{i = 1}^n u_i^2 \right) \left( \sum_{i = 1}^n v_i^2 \right) - \left( \sum_{i = 1}^n u_i v_i \right)^2 = \sum_{1 \leq i < j \leq n} (u_i v_j - u_j v_i)^2.\]

\begin{lemma}\label{lem:by_lagrange}
For any homogeneous metric space \(X = \{x_1, \ldots, x_n\}\) we have
\begin{equation}\label{eq:lagrange}
N_X''(1) + N_X'(1) - N_X'(1)^2 = \frac{1}{n^2} \sum_{1 \leq i < j \leq n} (d(x_1,x_i) - d(x_1,x_j))^2.
\end{equation}
If \(n > 1\) then this value is strictly positive.
\end{lemma}

\begin{proof}
For each \(i \in \{1,\ldots, n\}\) let \(d_i = d(x_1,x_i)\). By \Cref{lem:derivatives}, the left hand side of (\ref{eq:lagrange}) is
\begin{align*}
\frac{1}{n} \sum_{i=1}^n (d_i^2 - d_i) + \frac{1}{n} \sum_{i=1}^n d_i - \left(\frac{1}{n} \sum_{i=1}^n d_i\right)^2 &= \frac{1}{n^2} \left( n \sum_{i=1}^n d_i^2 - \left( \sum_{i=1}^n d_i\right)^2 \right) \\
&= \frac{1}{n^2} \left( \left(\sum_{i=1}^n 1^2\right) \left(\sum_{i=1}^n d_i^2\right) - \left( \sum_{i=1}^n 1 \cdot d_i \right)^2 \right) \\
&= \frac{1}{n^2} \sum_{1 \leq i < j \leq n} (d_j - d_i)^2
\end{align*}
where in the third line we use Lagrange's identity. To see that this value is strictly positive when \(n > 1\), we just need to see that for some \(i < j\) the term \(d_j - d_i\) is non-zero. Take \(i = 1\) and any \(j > 1\): then \(d_1 = d(x_1,x_1) = 0\) while \(d_j = d(x_1, x_j) \neq 0\), so \(d_j - d_1 > 0\).
\end{proof}

\begin{proof}[Proof of \Cref{thm:mag1_finite}]
\Cref{thm:fmag_hom_join} tells us that
\begin{equation}\label{eq:Mag(1)_join_1}
    \lim_{t \to 0} |t(X \ast Y)| = \lim_{q \to 1} \frac{N_X(q) + N_Y(q) - 2q}{N_X(q)N_Y(q) - q^2}
\end{equation}
if the limit on the right exists. Since \(N_X(1) = 1 = N_Y(1)\), both denominator and numerator in \eqref{eq:Mag(1)_join_1} converge to $0$ as $q\to 1$. Applying L'H\^opital's rule yields
\begin{equation}\label{eq:Mag(1)_join_2}
\lim_{t\to 0}|t(X*Y)|=
\lim_{q\to 1}
\frac{N'_X(q) + N'_Y(q) -2}{N'_X(q) N_Y(q)+N_X(q)N'_Y(q)-2q}, 
\end{equation}
provided this limit exists. Assumption \eqref{eq:assumption} says that \(N_X'(1) + N_Y'(1) = 2\), which ensures that both the denominator and numerator in \eqref{eq:Mag(1)_join_2} again go to 0 as \(q \to 1\), so we apply L'H\^opital's rule a second time to see that
\begin{equation}\label{eq:Mag(1)_join_3}
\lim_{t\to 0}|t(X*Y)| = \lim_{q\to 1} \frac{N''_X(q) + N''_Y(q)}{N''_X(q)N_Y(q)+2N'_X(q)N'_Y(q)+N_X(q)N''_Y(q)-2}
\end{equation}
provided \emph{this} limit exists.

By \Cref{lem:derivatives}, as \(q \to 1\) the numerator in \eqref{eq:Mag(1)_join_3} converges to
\[\frac{1}{n} \sum_{i=1}^n \left( d(x_1,x_i)^2 - d(x_1,x_i) \right) + \frac{1}{m} \sum_{i=1}^m \left( d(y_1,y_i)^2 - d(y_1,y_i) \right).\]
Meanwhile, since \(N_Y(1) = 1 = N_X(1)\), the denominator converges to
\[N''_X(1)+2N'_X(1)N'_Y(1)+N''_Y(1)-2,\]
which, by assumption \eqref{eq:assumption}, is equal to
\begin{align*}
& N_X''(1)+N_X'(1)+
N_Y''(1)+N_Y'(1)+2N_X'(1)N_Y'(1)-4\\
= \:
& N_X''(1)+N_X'(1)+
N_Y''(1)+N_Y'(1)+2N_X'(1)N_Y'(1)-(N_X'(1)+N_Y'(1))^2\\
= \:
& N_X''(1)+N_X'(1)-N_X'(1)^2+
N_Y''(1)+N_Y'(1)-N_Y'(1)^2\\
= \:
& \frac{1}{n^2}\sum_{1\leq i<j\leq n} (d(x_1, x_i)-d(x_1, x_j))^2+
\frac{1}{m^2}\sum_{1\leq i<j\leq m} (d(y_1, y_i)-d(y_1, y_j))^2
\end{align*}
where in the final line we use \Cref{lem:by_lagrange}. Assumption \eqref{eq:assumption} also ensures that at least one of \(n\) and \(m\) is greater than 1, so the same lemma tells us that the limiting value of the denominator is strictly positive. It follows that the limit on the right of \eqref{eq:Mag(1)_join_3} does exist, giving the explicit formula
\[\lim_{t\to 0}|t(X*Y)| = \frac
{\frac{1}{n} \sum_{i=1}^n \left( d(x_1,x_i)^2 - d(x_1,x_i) \right) + \frac{1}{m} \sum_{i=1}^m \left( d(y_1,y_i)^2 - d(y_1,y_i) \right)}
{\frac{1}{n^2} \sum_{i < j} (d(x_1,x_i) - d(x_1,x_j))^2 + \frac{1}{m^2} \sum_{i < j} (d(y_1,y_i) - d(y_1,y_j))^2}.\]

Finally, using the fact that $2=\frac{1}{2}(N'_X(1)+N'_Y(1))^2$, we have, from \eqref{eq:Mag(1)_join_3}, that
\begin{align}
    \label{eq:limitformula}
    \lim_{t\to 0}|t(X*Y)| &= \frac{N_X''(1) + N_Y''(1)}{N_X''(1) + N_Y''(1) - \frac{1}{2}(N_X'(1) - N_Y'(1))^2}
\end{align}
from which we see that the numerator is not less than the denominator. Hence, $1\leq \lim_{t\to 0}|t(X*Y)|<\infty$. \end{proof}

From \Cref{thm:mag1_finite} we can derive our final result.

\begin{theorem}\label{thm:arbitrary_values}
    For every real number \(R \geq 1\) there exists a finite metric space \(X\) such that \(\lim_{t \to 0} |tX| = R\).
\end{theorem}

\begin{proof}
For each natural number \(n > 1\) and real number \(r > 0\), let \(X_r^n\) denote the \(n\)-point metric space with \(d(x,x') = r\) for all \(x \neq x'\). Let \(s(n,r) = \frac{2n}{n-1} - r\). Then for \(r \in [2/(n-1), 2]\), both \(X_r^n\) and \(X_{s(n,r)}^n\) are homogeneous spaces of diameter at most 2, and the join \(X_r^n \ast X_{s(n,r)}^n\) satisfies the conditions of \Cref{thm:mag1_finite}.

For fixed \(n > 1\), we can compute \(\lim_{t \to 0}|t(X_r^n \ast X_{s(n,r)}^n)|\) as a function of \(r\). Explicitly, by \Cref{thm:mag1_finite} we have
\begin{align*}
\lim_{t \to 0}|t(X_r^n \ast X_{s(n,r)}^n)| &= \frac{\frac{n-1}{n}(r^2 - r) + \frac{n-1}{n} \left(\left(\frac{2n}{n-1} - r\right)^2 + r \right) - 2}{\frac{n-1}{n^2}r^2 + \frac{n-1}{n^2}\left(\frac{2n}{n-1} - r\right)^2}. \label{eq:mag1_rational}
\end{align*}
This is a rational function in \(r\), and is non-singular for \(r \in [2/(n-1), 2]\). In particular, it is continuous on the subinterval \([n/(n-1), 2]\). At the lower bound of this interval, taking \(r = n/(n-1)\) gives
\[\lim_{t \to 0}|t(X_r^n \ast X_{s(n,r)}^n)| = 1\]
while, at the upper bound, taking $r=2$ gives
\begin{equation}\label{eq:up_bd}
    \lim_{t \to 0}|t(X_2^n \ast X_{s(n,2)}^n)| = \frac{\frac{1}{2} n^{3} - \frac{3}{2} n^{2} + 2 n}{n^{2} - 2 n + 2}.
\end{equation}
Since the limit is a continuous function of \(r \in [n/(n-1), 2]\), the intermediate value theorem implies it must take every value in the range \([1, \lim_{t \to 0}|t(X_2^n \ast X_{s(n,2)}^n)|]\).

Now, take any real number \(R \geq 1\). For large \(n\) the quotient in \eqref{eq:up_bd} is close to \(n/2\), so for \(N \gg 2R\) we have \(\lim_{t \to 0}|t(X_2^N \ast X_{s(n,2)}^N)| > R\). Hence, for some \(r \in [N/(N-1), 2]\) we must have \(\lim_{t \to 0}|t(X_r^N \ast X_{s(N,r)}^N)| = R\).
\end{proof}

\begin{remark}
    Willerton's example belongs to the family of spaces constructed in the proof of \Cref{thm:arbitrary_values}: it is the join \(X_2^3 \ast X_1^3\). That particular space is of negative type; we do not know whether this holds for all members of the family.
\end{remark}

\bibliographystyle{amsplain}

\bibliography{RY_1-point.bib}

\end{document}